\documentclass[11pt,a4paper]{article}
\usepackage{amsmath,amsfonts,enumerate,amssymb,graphicx}
\usepackage[dvips]{epsfig}
\usepackage{pgf,tikz}
\usepackage{graphicx,float,anysize}
\usepackage{array,color}
\usepackage{amsmath,amssymb,amsfonts,amsthm,float,color,enumerate}
\usepackage{mathtools}
\usepackage{scrextend}
\usepackage{color}

\newtheorem{lem}{Lemma}[section]

\newtheorem{theorem}[lem]{Theorem}
\newtheorem{proposition}[lem]{Proposition}
\newtheorem{lemma}[lem]{Lemma}
\newtheorem{corollary}[lem]{Corollary}
\newtheorem{definition}[lem]{Definition}
\newtheorem{problem}[lem]{Problem}

\newtheorem{observation}[lem]{Observation}

\DeclarePairedDelimiter\ceil{\lceil}{\rceil}

\newcounter{case}

\renewcommand{\thecase}{\arabic{case}}

\newcounter{subcase}

\numberwithin{subcase}{case}

\newcounter{subsubcase}

\numberwithin{subsubcase}{subcase}

\begin{document}
\begin{center}
{\bf\large Product irregularity strength of graphs with small clique cover number} \\ [+4ex]

Daniil Baldouski
\\ [+2ex]
{\it University of Primorska, FAMNIT, Glagolja\v ska 8, 6000 Koper, Slovenia}
\end{center}

\begin{abstract}
For a graph $X$  without isolated vertices and without isolated edges, a {\em product-irregular labelling} $\omega:E(X)\rightarrow \{1,2,\ldots,s\}$, first defined by Anholcer in 2009, is a labelling of the edges of $X$ such that 
for any two distinct vertices $u$ and $v$ of $X$ the product of labels of the edges incident with $u$ is
different from the product of labels of the edges incident with $v$.
The minimal $s$ for which there exist a product irregular labeling is called  {\em the product irregularity strength} of $X$ and is denoted by $ps(X)$. {\it Clique cover number} of a graph is the minimum number of cliques that partition its vertex-set. In this paper we prove that connected graphs with clique cover number $2$ or $3$ have the product-irregularity strength equal to $3$, with some small exceptions.
\end{abstract}

\begin{quotation}
\noindent {\em Keywords: product irregularity strength, clique-cover number.}\\
\noindent {\em Math. Subj. Class.: 05C15, 05C70, 05C78}
\end{quotation}

\section{Introduction}
\noindent

Throughout this paper let $X$ be a simple graph, that is, a graph without loops or multiple edges, without isolated vertices and without isolated edges. Let $V(X)$ and $E(X)$ denote the vertex set and the edge set of $X$, respectively. 
Let $\omega :E(X)\rightarrow \{1,2,\ldots, s\}$ be an integer labelling  of the edges of $X$. Then the {\em product degree} $pd_X(v)$ of a vertex $v\in V(X)$ in the graph $X$ with respect to the labelling $\omega$ is defined by
$$
pd_X(v)= \prod_{v\in e} \omega(e).
$$
If the graph $X$ is clear from the context, then we will simply use $pd(v)$.
A labelling $\omega$ is said to be {\it product-irregular}, if any two distinct vertices $u$ and $v$ of $X$ have different corresponding product degrees, that is, $pd_X(u)\ne pd_X(v)$ for any $u$ and $v$ in $V(X)$ ($u\neq v$).
The {\it product irregularity strength} $ps(X)$ of $X$ is the smallest positive integer  $s$ for which there exists a product-irregular labelling $\omega :E(X)\rightarrow \{1,2,\ldots, s\}$.

This concept was first introduced by Anholcer in \cite{A09} as a multiplicative version of the well-studied concept of irregularity strength of graphs introduced by Chartrand et al. in \cite{CJLORS88} and studied later quite extensively (see for example \cite{B, F, HP, P}).
A concept similar to product-irregular labelling is the {\em product anti-magic labeling} of a graph, where it is required that the labeling $\omega$ is bijective (see \cite{K,P07}). It is clear that every product anti-magic labeling is product-irregular. Another related concept is the so-called {\it multiplicative vertex-colouring} (see \cite{SK12,SK17}), where it is required that $pd(u)\neq pd(v)$ for every pair of adjacent vertices $u$ and $v$, while non-adjacent vertices can have the same product degrees. It is easy to see that every product-irregular labelling is a multiplicative vertex-colouring.

 In \cite{A09} Anholcer gave upper and lower bounds on product irregularity strength of graphs. The main results in \cite{A09} are estimates for product irregularity strength of cycles, in particular it was proved that for every $n > 2$
\begin{equation*}
ps(C_n) \geq \ceil{\sqrt{2n} - \frac{1}{2}},
\end{equation*}
and that for every $\varepsilon > 0$ there exists $n_0$ such that for every $n \geq n_0$

\begin{equation*}
ps(C_n) \leq \ceil{(1 + \varepsilon) \sqrt{2n} \mbox{ ln } n}.
\end{equation*}

Anholcer in \cite{A14} considered product irregularity strength of complete bipartite graphs and proved that for two integers $m$ and $n$ such that $2 \geq m \geq n$
it holds
$ps(K_{m,n}) = 3$
if and only if $n \geq$ ${m+2}\choose{2}$.

In \cite{DH15}, Darda and Hujdurovi\' c proved that for any graph $X$ of order at least $4$ with at most one isolated vertex and without isolated edges we have $ps(X) \leq |V(X)| - 1$. Connections between product irregularity strength of graphs and multidimensional multiplication table problem was established, see \cite{F08,K10} for some results on multidimensional multiplication problem.

It is easy to see that the lower bound for the product irregularity strength of any graph is $3$. In this paper we will give some sufficient conditions for a graph to have product irregularity strength equal to $3$. In particular we will prove that graphs of order greater than $6$ with clique-cover number $2$ have product irregularity strength $3$ (see Corollary \ref{cmain1}), where {\it clique cover number} is the minimum number of cliques that partition the vertex set. Moreover, we will prove that for a connected graph such that its vertex set can be partitioned into $3$ cliques of sizes at least $4$ then its product irregularity strength  is $3$ (see Corollary \ref{cmain2}).

The paper is organized as follows. In section 2 we rephrase the definition of product-irregular labellings in terms of the corresponding weighted adjacency matrices and give some constructions that will be used for proving our main results. In section 3 we will determine the product irregularity strength of graphs with clique cover number 2, while in section 4 we study product irregularity strength of graphs with clique cover number 3.

\section{Product-irregular matrices}

In this section we will rephrase the definition of product irregular labelling of graphs using weighted adjacency matrices. We start with the definition of weighted adjacency matrix.

\begin{definition}
Let $w$ be an integer labelling of the edges of a graph $X$ of order n with $V(X) = \{ v_1, v_2, \ldots , v_n \}$. Weighted adjacency matrix of X is $n \times n$ matrix $M$ where $M_{ij} = w(\{ v_i, v_j \})$ if $v_i$ and $v_j$ are adjacent and $M_{ij} = 0$ otherwise.
\end{definition}

\begin{definition}[Product-irregular matrices and product degree for matrices] Assume that we have weighted adjacency $n \times n$ matrix $M$ ($n \geq 2$). Then $pd(M_{v_k}) := {\displaystyle \prod_{M_{k,i} \neq 0} M_{k,i}}$ is the product of all non-zero elements of the row $v_k$. We say that $M$ is product-irregular if $\forall i,j \in \{ 1,2, \ldots, n \}$ for $i \neq j$ $pd(M_{v_i}) \neq pd(M_{v_j})$. 
We will work with matrices with entries $a_{ij} \in \{0,1,2,3\}$ therefore  to simplify reading  for $v \in M$   if $pd(v)=2^{a} \cdot 3^{b}$  then we will use notation $pd(v) := (a, b)$. Also define $pd(v)[1] := a$ and $pd(v)[2] := b$.
\end{definition}

\begin{observation} Labelling of graph is product-irregular if and only if the corresponding weighted adjacency matrix is product-irregular.
\end{observation}

Let $n \geq 4$ and let $M_n(x,y,z)$ be $n \times n$ matrix such that $M_n(x,y,z) = (m_{ij})$ where
$$
m_{ij} = 
\begin{cases} 
0, & $if $ i=j $ $\\ 
x, &$if $ j \leq n-i+1 $ and $ i \neq j\\ 
z, & $if $ (i,j) = (k,n)$ or $(i,j) = (n,k)$ for $k = \ceil{\frac{n}{2}} + 1\\ 
y, &$otherwise $ 
\end{cases}
$$
\\For example:
\begin{equation}
M_7(x, y, z) = \begin{pmatrix}
0 & x & x & x & x & x & x \\
x & 0 & x & x & x & x & y \\
x & x & 0 & x & x & y & y \\
x & x & x & 0 & y & y & y \\
x & x & x & y & 0 & y & z \\
x & x & y & y & y & 0 & y \\
x & y & y & y & z & y & 0
\end{pmatrix}.
\end{equation}
We will denote with $A \oplus B$ the direct sum of matrices $A$ and $B$, that is 
$$
A \oplus B = \begin{pmatrix}
A & 0 \\
0 & B
\end{pmatrix},
$$
where $0$ denotes zero matrix of appropriate size.
\subsection{Properties of $M_n$}
Let $x_i$, $y_i$ and $z_i$ be the number of $x$, $y$ and $z$ respectively appearing in the $i$-th row of matrix $M_n(x,y,z)$. For fixed $n$ with $k$ we denote $k := \ceil{\frac{n}{2}} + 1 $. Then the rows of the matrix $M_n(x,y,z)$ can be separated into 3 types:
\begin{labeling}{alligator}
\item[1st type:] $(x_k,y_k,z_k) = (\ceil{\frac{n-1}{2}}, \ceil{\frac{n}{2}} - 2, 1)$,
\item[2nd type:] $(x_i,y_i,z_i) = (n-i, i-1, 0)$ for $i<k$ and $(x_i,y_i,z_i) = (n-i+1, i-2, 0)$ for $n>i>k$,
\item[3rd type:] $(x_n,y_n,z_n) = (1, n-3, 1)$.
\end{labeling}
We denote by $m_{(i)}(M)$ the row of type $i$ for $i \in \{ 1, 2, 3 \}$ of matrix $M$, where $M$ is matrix $M_n(x,y,z)$ (if the matrix $M_n(x,y,z)$ is clear from the context, then we will simply use $m_{(i)}$). We start by proving the following nice property of matrix $M_n$.

\begin{proposition} 
If $\{ x,y,z \}$ is a set of distinct pairwise relatively prime integers, then $M_n(x,y,z)$ is product irregular matrix for any $n \geq 4$.
\end{proposition}

\begin{proof} Suppose contradiction, i.e. there exist $m_i$ and $m_j$ (that are rows of matrix $M_n(x,y,z)$) for some $i \neq j$ such that $pd(m_i) = pd(m_j)$. There are 3 types of rows therefore it is enough to check the equality above not for all rows, but for all types of rows. Observe that for every $i \in \{ 1,2,\ldots,n \}$ the sum $x_i + y_i + z_i = n - 1$ and $pd(m_i) = pd(m_j)$ for some $i \neq j$ if and only if $x_i = x_j$, $y_i = y_j$ and $z_i = z_j$. It follows that:
\begin{enumerate}
\item If $pd(m_{(1)}) = pd(m_{(3)}) \Rightarrow (\ceil{\frac{n-1}{2}}, \ceil{\frac{n}{2}} - 2, 1) = (1, n-3, 1) \Rightarrow n = 3$ which is contradiction.
\item Since rows of second type have value 0 at 3rd coordinate and rows of first and third types have value 1 at 3rd coordinate, then $pd(m_{(2)}) \neq pd(m_{(i)})$ for $i \in \{ 1,3 \}$.
\item It is clear that $(x_i,y_i,z_i) \neq (x_j,y_j,z_j)$ for $i<k$ and $k<j<n$ i.e. product degrees of different rows of type 2 are different.
\end{enumerate}
We were considering different rows, that means we did not have to consider $pd(m_{(i)}) = pd(m_{(i)})$ for every $i \in \{ 1,3 \}$.
\end{proof} 

We will define 3 matrices of class $M_n(x,y,z)$ for specific $x, y$ and $z$. Assign matrix $A_n := M_n(1,2,3)$, $B_n := M_n(2,3,1)$ and $C_n := M_n(3,1,2)$.

\subsection{Properties of $A_n \oplus B_m$}

\begin{lemma} \label{l24}
For every $m \geq n \geq 4$, $A_n \oplus B_m$ is product irregular if $(n,m) \not\in \{ (4,4), (5,5), (6,6) \}$.
\end{lemma}

\begin{proof}
Suppose contradiction, i.e. there exist $a_i$ and $b_j$ (that are rows of matrices $A_n$ and $B_m$ respectively) for some i and j such that $pd(a_i) = pd(b_j)$. There are 3 types of rows therefore it is enough to check all of the 9 possibilities for different types of rows:
\begin{enumerate}
\item if $pd(a_{(1)}) = pd(b_{(1)}) \Rightarrow (\ceil{\frac{n}{2}} - 2, 1) = (\ceil{\frac{m-1}{2}}, \ceil{\frac{m}{2}} - 2)$ which contradicts with $m \geq n$.
\item if $pd(a_{(1)}) = pd(b_{(2)}) \Rightarrow (\ceil{\frac{n}{2}} - 2, 1) = (m - j, j - 1)$ or $(\ceil{\frac{n}{2}} - 2, 1) = (m - j + 1, j - 2)$ which contradicts with $m \geq n \geq 4$.
\item if $pd(a_{(1)}) = pd(b_{(3)}) \Rightarrow (\ceil{\frac{n}{2}} - 2, 1) = (1, m - 3) \Rightarrow (n,m) = (4,4)$ or $(5,4)$  which is contradiction.
\item if $pd(a_{(2)}) = pd(b_{(1)}) \Rightarrow (i - 1, 0) = (\ceil{\frac{m-1}{2}}, \ceil{\frac{m}{2}} - 2)$ or $(i - 2, 0) = (\ceil{\frac{m-1}{2}}, \ceil{\frac{m}{2}} - 2) \Rightarrow m=3$ or $4 \Rightarrow m=4 \Rightarrow n>i$ which is contradiction.
\item if $pd(a_{(2)}) = pd(b_{(2)})$. We have that $pd(a_{(2)})[2] = 0 \Rightarrow pd(b_{(2)})[2] = 0 \Rightarrow pd(b_{(2)})[1] > pd(a_{(2)})[1]$ which is contradiction.
\item if $pd(a_{(2)}) = pd(b_{(3)}) \Rightarrow pd(a_{(2)})[2] = 0$ and $pd(b_{(3)})[2] > 0$ which is contradiction.
\item if $pd(a_{(3)}) = pd(b_{(1)}) \Rightarrow (n-3, 1) =  (\ceil{\frac{m-1}{2}}, \ceil{\frac{m}{2}} - 2) \Rightarrow (n,m) = (5,5)$ or $(6,6)$ which is contradiction.
\item if $pd(a_{(3)}) = pd(b_{(2)}) \Rightarrow (n - 3, 1) = (m-j, j-1)$ or $(n - 3, 1) = (m-j+1, j-2) \Rightarrow n > m$ in both cases which is contradiction.
\item if $pd(a_{(3)}) = pd(b_{(3)}) \Rightarrow (n - 3, 1) = (1 , m - 3) \Rightarrow (n, m) = (4, 4)$ which is contradiction.
\end{enumerate}
This finishes the proof.
\end{proof}

For the next lemma we need to consider weighted adjacency matrix 
\begin{equation} \label{eq:T3}
T := 
\begin{pmatrix}
0 & 1 & 2 \\ 
1 & 0 & 3 \\ 
2 & 3 & 0 
\end{pmatrix}
\end{equation}
Observe that $pd(T_1) = (1,0)$, $pd(T_2) = (0,1)$, $pd(T_3) = (1,1) \Rightarrow ps(K_3) = 3$.
\begin{lemma} \label{l234}
Let T be the matrix defined in \eqref{eq:T3} then for every $n \geq 5$ $T \oplus B_n$ is product irregular.
\end{lemma}

\begin{proof}
Observe that $\{ pd(T_i)$ for $i \in \{ 1,2,3 \} \} \subset \{ pd((A_4)_i)$ for $i \in \{ 1,2,3,4 \} \}$ and we know from Lemma \ref{l24} that $\forall n \geq 5$ $A_4 \oplus B_n$ is product irregular.
\end{proof}

\section{Graphs with clique-cover number 2}

In this section we consider product irregularity strength of connected graphs with clique cover number two. 
Suppose that $G$ is a graph with clique-cover number $2$, that is the vertex set of $G$ can be partitioned into two cliques $C_1$ and $C_2$, of sizes $n$ and $m$ respectively. Then it follows that $G$ has a spanning subgraph isomorphic to $K_n+K_m$, where for two graphs $H_1$ and $H_2$, $H_1+H_2$ denotes the disjoint union of $H_1$ and $H_2$. 
Then by \cite[Lemma 1]{DH15} it follows that $3\leq ps(G)\leq ps(K_n+K_m)$.
Hence we will start by considering product irregularity strength of $K_n+K_m$.

It can be proved that any $4 \times 4$ weighted adjacency matrix $M$ (with weights $1, 2$ and $3$) is product irregular if and only if there exist row $m \in M$ such that $pd(m) = (1,1)$. Therefore $ps(K_4 + K_4) > 3$. There are a lot of graphs of the form $K_n + K_m$ for some integers $n$ and $m$ with product irregularity strength greater than $3$. But since such graphs are disconnected, we will define operation of adding an edge between components of these graphs, i.e. we will consider minimal connected graphs with clique cover number $2$.

\begin{definition}[$+edge$]
Let $G_1$ and $G_2$ be two graphs with disjoint vertex sets. With $G_1 + G_2 + edge$ we denote a graph obtained by taking disjoint union of $G_1$ and $G_2$ and adding an edge between two vertices of $G_1$ and $G_2$. 
\end{definition}

\begin{lemma} \label{l224}
$\forall n \geq 4$, $ps(K_2 + K_n + edge)=3$.
\end{lemma}

\begin{proof}
Consider weighted adjacency $(n+2) \times (n+2)$ matrix
\begin{equation} \label{eq:L}
L = \begin{pmatrix} 
0 & 1 & 3 & 0 & \cdots & 0 \\ 
1 & 0 & 0 & 0 & \cdots & 0 \\ 
3 & 0 \\
\vdots & \vdots \\
0 & 0 & & B_n \\
\vdots & \vdots \\ 
0 & 0 
\end{pmatrix}
\end{equation}
where $L_{1,3} = L_{3, 1} = 3$. Clearly, $L$ is weighted adjacency matrix of the graph $K_2 + K_n$. We will show that $L$ is product-irregular. Since we have that $pd((B_n)_i) = pd(L_{i+2})$ for every $i \in \{ 2, 3, \ldots n \}$ it is enough to show that product degrees of first $3$ rows of matrix $L$ are different and do not belong to the set $\{ pd((B_n)_i)$, $i \in \{ 2,3,\ldots,n \} \}$.
\begin{enumerate}
\item It is clear that those rows are different and that first two rows of $L$ are not in the set $\{ pd((B_n)_i)$, $i \in \{ 2,3,\ldots,n \} \}$.
\item For the row $L_3$ we have that
$pd(L_3) = pd((B_n)_1) \cdot 3 = (n-1, 1)$. Therefore $pd(L_3)[1] + pd(L_3)[2] = n - 1 + 1 > n-1 \geq pd((B_n)_j)[1] + pd((B_n)_j)[2]$ for any $j \in \{ 2,3, \ldots, n \}$.
\end{enumerate}
This finishes the proof.
\end{proof}

\begin{corollary} \label{c214}
$\forall n \geq 4$, $ps(K_1 + K_n + edge)=3$.
\end{corollary}

\begin{proof}
Consider matrix $L'$ obtained from matrix $L$ from \eqref{eq:L} by deleting second row and column. Clearly, $L'$ is product-irregular.
\end{proof}

\begin{theorem} \label{tmain1}
For every n and m that are greater or equal than 1 and such that $n + m > 6$ we have $ps(K_n + K_m + edge) = 3$.
\end{theorem}

\begin{proof}
Consider some cases that were not covered by previous Lemmas:
\begin{enumerate}
\item[(i)] $ps(K_5 + K_5) = 3$. For proving this fact we can take direct sum of the following weighted adjacency matrices: 
\begin{equation} \label{eq:T5}
T_5 :=
\begin{pmatrix} 
0 & 3 & 1 & 1 & 1 \\ 
3 & 0 & 1 & 3 & 2 \\ 
1 & 1 & 0 & 1 & 1 \\ 
1 & 3 & 1 & 0 & 2 \\ 
1 & 2 & 1 & 2 & 0 
\end{pmatrix} 
\mbox{   and   }
\tilde{T_5} :=
\begin{pmatrix}
0 & 2 & 2 & 2 & 1 \\
2 & 0 & 3 & 3 & 3 \\
2 & 3 & 0 & 2 & 3 \\
2 & 3 & 2 & 0 & 1 \\
1 & 3 & 3 & 1 & 0
\end{pmatrix}
\end{equation}
\item[(ii)] $ps(K_6 + K_6) = 3$. For proving this fact we can take direct sum of the following weighted adjacency matrices:
\begin{equation} \label{eq:T6}
T_6 :=
\begin{pmatrix}
0 & 1 & 2 & 3 & 1 & 3 \\
1 & 0 & 1 & 3 & 1 & 1 \\
2 & 1 & 0 & 1 & 2 & 2 \\
3 & 3 & 1 & 0 & 1 & 1 \\
1 & 1 & 2 & 1 & 0 & 1 \\
3 & 1 & 2 & 1 & 1 & 0
\end{pmatrix}
\mbox{   and   }
\tilde{T_6} :=
\begin{pmatrix}
0 & 2 & 3 & 3 & 3 & 3 \\
2 & 0 & 2 & 3 & 3 & 2 \\
3 & 2 & 0 & 2 & 1 & 2 \\
3 & 3 & 2 & 0 & 3 & 1 \\
3 & 3 & 1 & 3 & 0 & 3 \\
3 & 2 & 2 & 1 & 3 & 0
\end{pmatrix}
\end{equation}

\end{enumerate}

Also consider some cases that could not be proved without adding edges between cliques.
\begin{enumerate}
\item[(iii)] $ps(K_4 + K_4 + edge) = 3$. For proving this fact we will consider the following product-irregular matrix: 
\begin{equation}
\begin{pmatrix}
0 & 1 & 1 & 1 & 0 & 0 & 0 & 0 \\
1 & 0 & 1 & 2 & 0 & 0 & 0 & 0 \\
1 & 1 & 0 & 3 & 0 & 0 & 0 & 0 \\
1 & 2 & 3 & 0 & 3 & 0 & 0 & 0 \\
0 & 0 & 0 & 3 & 0 & 2 & 2 & 2 \\
0 & 0 & 0 & 0 & 2 & 0 & 2 & 3 \\
0 & 0 & 0 & 0 & 2 & 2 & 0 & 1 \\
0 & 0 & 0 & 0 & 2 & 3 & 1 & 0
\end{pmatrix}
\end{equation}
\item[(iv)] $ps(K_3 + K_4 + edge) = 3$ follows from the previous item (i) using the same proof as in the Lemma 3.
\end{enumerate}

The proof now follows by Lemmas \ref{l24}, \ref{l234} and \ref{l224} and Corollary \ref{c214}.

\end{proof}

\begin{corollary} \label{cmain1}
If $G$ is a connected graph of order at least 7 with clique-cover number $2$ then $ps(G) = 3$.
\end{corollary}

Observe that $ps(K_3+K_3+edge)=4$, i.e. 6 is the lower bound of the sum $n+m$ in the Theorem 1.

\section{Graphs with clique-cover number 3}
In this section we consider the product irregularity strength of graphs with clique-cover number $3$. Observe that a graph $G$ has clique cover number $3$, if and only if its complement has chromatic number equal to $3$. If $G$ is a graph with clique cover number tree, then its vertex set can be partitioned into three cliques, of sizes $n$, $m$ and $l$. Then it follows that $G$ has a spanning subgraph isomorphic to $K_n+K_m+K_l$, hence we will first investigate the product irregularity strength of such graphs.

\subsection{Properties of $A_n \oplus B_m \oplus C_l$}

\begin{lemma} \label{l27_4}
$\forall n \geq 7$ and $m \geq 4$, $A_n \oplus C_m$ is product irregular.
\end{lemma}

\begin{proof}
Suppose contradiction, i.e. $\exists a_i$ and $c_j$ (that are rows of matrices $A_n$ and $C_m$ respectively) for some i and j such that $pd(a_i) = pd(c_j)$. We will use the same type of proof as in the Lemma \ref{l24}.
\begin{enumerate}

\item if $pd(a_{(1)}) = pd(c_{(1)}) \Rightarrow (\ceil{\frac{n}{2}} - 2, 1) = (1, \ceil{\frac{m-1}{2}}) \Rightarrow n = 5$ or $6$ and $m = 2$ or $3$ which is contradiction.
\item if $pd(a_{(1)}) = pd(c_{(2)}) \Rightarrow (\ceil{\frac{n}{2}} - 2, 1) = (0, m - j)$ or $(\ceil{\frac{n}{2}} - 2, 1) = (0, m - j + 1)$. In both cases $n = 3$ or $4$ which is contradiction.
\item if $pd(a_{(1)}) = pd(c_{(3)}) \Rightarrow (\ceil{\frac{n}{2}} - 2, 1) = (1, 1) \Rightarrow n = 5$ or $6$ which is contradiction.
\item if $pd(a_{(2)}) = pd(c_{(1)}) \Rightarrow (i-1, 0) = (1, \ceil{\frac{m-1}{2}})$ or $(i-2, 0) = (1, \ceil{\frac{m-1}{2}})$. In both cases $m = 1$ which is contradiction.
\item for $pd(a_{(2)}) = pd(c_{(2)})$ we have that $pd(a_{(2)})[2] = 0$ and $pd(c_{(2)})[2] > 0$ which is contradiction.
\item if $pd(a_{(2)}) = pd(c_{(3)}) \Rightarrow (i-1, 0) = (1, 1)$ or $(i-2, 0) = (1, 1))$ which is, clearly, contradiction.
\item $pd(a_{(3)}) = pd(c_{(1)}) \Rightarrow (n-3, 1) = (1, \ceil{\frac{m-1}{2}}) \Rightarrow (n,m) = (4,3)$ or $(4,4)$ which is contradiction.
\item $pd(a_{(3)}) = pd(c_{(2)}) \Rightarrow (n-3, 1) = (0, m-j)$ or $(n-3, 1) = (0, m-j+1) \Rightarrow n = 3$ which is contradiction.
\item $pd(a_{(3)}) = pd(c_{(3)}) \Rightarrow (n-3, 1) = (1,1) \Rightarrow n=4$ which is contradiction.
\end{enumerate}
This finishes the proof.
\end{proof}

\begin{lemma} \label{l2_55}
$\forall n \geq m \geq 5$, $B_n \oplus C_m$ is product irregular if $(n,m) \not\in \{ (5,4), (5,5), (6,6) \}$.
\end{lemma}

\begin{proof}
Suppose contradiction, i.e. there exist $b_i$ and $c_j$ (that are rows of matrices $B_n$ and $C_m$ respectively) for some i and j such that $pd(b_i) = pd(c_j)$. We will use the same type of proof as in the Lemma \ref{l24}.
\begin{enumerate}
\item $pd(b_{(1)}) = pd(c_{(1)}) \Rightarrow (\ceil{\frac{n-1}{2}}, \ceil{\frac{n}{2}} - 2) = (1, \ceil{\frac{m-1}{2}}) \Rightarrow n = 2$ or $3$ which is contradiction.
\item $pd(b_{(1)}) = pd(c_{(2)}) \Rightarrow (\ceil{\frac{n-1}{2}}, \ceil{\frac{n}{2}} - 2) = (0, m - j)$ or $(\ceil{\frac{n-1}{2}}, \ceil{\frac{n}{2}} - 2) = (1, m - j + 1)$ which contradicts with $n \geq 7$.
\item $pd(b_{(1)}) = pd(c_{(3)}) \Rightarrow (\ceil{\frac{n-1}{2}}, \ceil{\frac{n}{2}} - 2) = (1,1)$ which is, clearly, contradiction.
\item if $pd(b_{(2)}) = pd(c_{(1)}) \Rightarrow (n-i, i-1) = (1, \ceil{\frac{n-1}{2}})$ or $(n-i+1, i-2) = (1, \ceil{\frac{n-1}{2}})$ which contradicts with $n > i$. 
\item for $pd(b_{(2)}) = pd(c_{(2)})$ we have that $pd(b_{(2)})[1] > 0$ and $pd(c_{(2)})[1] = 0$ which is contradiction.
\item if $pd(b_{(2)}) = pd(c_{(3)}) \Rightarrow (n-i, i-1) = (1,1)$ or $(n-i+1, i-2) = (1,1)$ which contradicts with $n > i$.
\item $pd(b_{(3)}) = pd(c_{(1)}) \Rightarrow (1, n-3) = (1, \ceil{\frac{m-1}{2}}) \Rightarrow m = 2(n-3)$ or $m = 2(n-3) + 1$ which is contradiction because for $n \geq 7$ we have that $m > n$ and for $4 \leq n < 7$ we have that $(n,m) \in \{ (5,4), (5,5), (6,6) \}$.
\item $pd(b_{(3)}) = pd(c_{(2)}) \Rightarrow (1, n-3) = (0, m-j)$ or $(1, n-3) = (0, m-j+1)$ which is contradiction.
\item $pd(b_{(3)}) = pd(c_{(3)}) \Rightarrow (1, n-3) = (1,1) \Rightarrow n=4$ which is contradiction.
\end{enumerate}
This finishes the proof.
\end{proof}

\begin{theorem} \label{t37}
For every n, m and l such that $m \geq l \geq n \geq 7$ $A_n \oplus B_m \oplus C_l$ is product irregular.
\end{theorem}

\begin{proof}
Proof follows by Lemmas \ref{l24}, \ref{l27_4} and \ref{l2_55}.
\end{proof}

\begin{corollary} \label{c37}
For all positive integers $n,m$ and $l$ greater or equal than $7$ it holds that $ps(K_n + K_m + K_l) = 3$.
\end{corollary}

\begin{lemma} \label{l366}
For all positive integers $n$ and $m$ greater than $6$ and $k \in \{ 4,5,6 \}$, $ps(K_n + K_m + K_k)=3$.
\end{lemma}

\begin{proof}
Let $m \geq n$ and consider matrix $A_n \oplus B_m \oplus C_k$. From Lemmas 1, 4 and 5 we can conclude that this matrix is product-irregular.
\end{proof}

\begin{lemma} \label{l_667}
For all positive integer $n \geq 7$ $ps(K_6 + K_6 + K_n)=3$.
\end{lemma}

\begin{proof} 
Consider $T_6 \oplus \tilde{T_6} \oplus B_n$ which is product-irregular because for every row $b$ of matrix $B_n$ $pd(b)[1] + pd(b)[2] \geq 6$ except $n = 7$ for which $pd(b_{(1)}) = (3, 2)$. But $(3, 2)$ does not belong to the set of all product degrees of matrix $T_6 \oplus \tilde{T_6}$.
\end{proof}

\begin{lemma} \label{l3567}
For all positive integer $n \geq 7$ $ps(K_5 + K_6 + K_n)=3$.
\end{lemma}

\begin{proof}
Consider the following matrix:
\begin{equation}
\begin{pmatrix}
0 & 2 & 2 & 2 & 1 & 1 \\
2 & 0 & 3 & 3 & 3 & 1 \\
2 & 3 & 0 & 2 & 3 & 1 \\
2 & 3 & 2 & 0 & 1 & 2 \\
1 & 3 & 3 & 1 & 0 & 1 \\
1 & 1 & 1 & 2 & 1 & 0
\end{pmatrix}
\oplus
\begin{pmatrix}
0 & 3 & 1 & 1 & 1 \\
3 & 0 & 1 & 3 & 2 \\
1 & 1 & 0 & 1 & 1 \\
1 & 3 & 1 & 0 & 2 \\
1 & 2 & 1 & 2 & 0
\end{pmatrix}
\oplus B_n
\end{equation}
is product-irregular because of the same proof as in Lemma \ref{l_667}.
\end{proof}

\begin{lemma} \label{l3556}
For all positive integers $n \geq 6$, $ps(K_5 + K_5 + K_n)=3$.
\end{lemma}

\begin{proof}
Consider weighted adjacency matrices $T_5$ and $\tilde{T_5}$ from \eqref{eq:T5} in the first item of the proof of Theorem 1:
\begin{enumerate}
\item $\forall n \geq 7$ we have $T_5 \oplus \tilde{T_5} \oplus B_n$ is product irregular because for every row $b$ of matrix $B_n$ $pd(b)[1] + pd(b)[2] \geq 5$.
\item For $n=6$ we have that $T_5 \oplus \tilde{T_5} \oplus P_6$ is product-irregular, where
\begin{equation}
P_6 :=
\begin{pmatrix}
0 & 2 & 2 & 2 & 2 & 1 \\
2 & 0 & 2 & 2 & 2 & 3 \\
2 & 2 & 0 & 2 & 3 & 3 \\
2 & 2 & 2 & 0 & 3 & 1 \\
2 & 2 & 3 & 3 & 0 & 3 \\
1 & 3 & 3 & 1 & 3 & 0
\end{pmatrix}.
\end{equation}
\end{enumerate}
This finishes the proof.
\end{proof}

It can be proved that $ps(K_5 + K_5 + K_4) = 4$. There are a lot of graphs of the form $K_n + K_m + K_k$ for some integers $n$, $m$ and $k$ with product irregularity strength greater than $3$. But since such graphs are disconnected, we will define operation of adding $2$ edges between components of these graphs such that the resulting graph will be connected, i.e. we will consider minimal connected graphs with clique cover number $3$.

\begin{definition}[$+2edges$]
Let $+2dges$ for graphs $G_1 + G_2 + G_3$ be the operation of adding edges, i.e. applying two times $+edge$ between any 2 different pairs of different sets $V(G_1), V(G_2)$ and $V(G_3)$. We will use the following notation for that operation: $G_1 + G_2 + G_3 + 2edges$.
\end{definition}

Now we will describe this operation using matrix language. Consider weighted adjacency matrices $A,B,C$ of sizes $n \times n$, $m \times m$ and $l \times l$ respectively. Let $T_{12}(A,B,C,i,j,w)$ be $(n+m+l) \times (n+m+l)$ matrix with all zeros except elements with coordinates $(i,n+j)$ and $(n+j, i)$ of value $w$. In a similar way we can define matrices $T_{13}(A,B,C,i,j,w)$ and $T_{23}(A,B,C,i,j,w)$ for which coordinates of non-zero elements are $(i, n+m+j)$ and $(n+m+j,i)$ and $(n+i, n+m+j)$ and $(n+m+j, n+i)$ respectively.

For example one of the weighted adjacency matrices for graph $K_n + K_m + K_l + 2edges$ where the edges between cliques are between vertices $a_i$ and $b_j$ of weight $w_1$ and between vertices $b_j$ and $c_k$ of weight $w_2$ where $a_i \in V(K_n), b_j \in V(K_m)$ and $c_k \in V(K_l)$ is $A_n \oplus B_m \oplus C_l + T_{12}(A_n,B_m,C_l,i,j,w_1) + T_{23}(A_n,B_m,C_l,j,k,w_2)$.

\begin{definition}[In-degree and in-edges]
Consider graph $G := G_1 + G_2 + G_3 +2edges$. Let $G^\prime := G_1 + G_2 + G_3$ be subgraph of the graph $G$. Let $g \in V(G)$ and let $d_0(g)$ to be the degree of vertex $g \in V(G^\prime)$. Then define \textit{in-degree} of vertex $g \in V(G)$ to be $d^+(g) := d(g) - d_0(g)$. We say that for some $i \in \{ 1,2,3 \}$ $G_i$ has $t$ \textit{in-edges} if and only if $$\sum_{g \in V(G_i)}{d^+(g)} = t.$$
\end{definition}

For the next theorem we will define the following matrix. Let $\tilde{M}_n(x,y):= M_n(x,y,y)$ and matrices $\tilde{A}_n, \tilde{B}_n$ and $\tilde{C}_n$ to be $\tilde{M}_n(1,2), \tilde{M}_n(2,3)$ and $\tilde{M}_n(3,1)$ respectively.

\begin{theorem} \label{t35e}
For all positive integers n, m and l that are greater or equal than 5 we have that $ps(K_n + K_m + K_l + 2edges) = 3$.
\end{theorem}

\begin{proof} Consider some cases that were not covered by previous Lemmas:
\begin{enumerate}
\item For $(n,m,l) = (6,6,6)$ consider the following matrix:
\begin{equation} \label{eq:M666}
\begin{pmatrix}
0 & 1 & 1 & 1 & 1 & 1 \\
1 & 0 & 3 & 1 & 1 & 2 \\
1 & 3 & 0 & 1 & 2 & 2 \\
1 & 1 & 1 & 0 & 2 & 2 \\
1 & 1 & 2 & 2 & 0 & 2 \\
1 & 2 & 2 & 2 & 2 & 0
\end{pmatrix}
\oplus
\begin{pmatrix}
0 & 2 & 2 & 2 & 2 & 2 \\
2 & 0 & 1 & 2 & 2 & 3 \\
2 & 1 & 0 & 2 & 3 & 3 \\
2 & 2 & 2 & 0 & 3 & 3 \\
2 & 2 & 3 & 3 & 0 & 3 \\
2 & 3 & 3 & 3 & 3 & 0
\end{pmatrix}
\oplus
\begin{pmatrix}
0 & 3 & 3 & 3 & 3 & 3 \\
3 & 0 & 2 & 3 & 3 & 1 \\
3 & 2 & 0 & 3 & 1 & 1 \\
3 & 3 & 3 & 0 & 1 & 1 \\
3 & 3 & 1 & 1 & 0 & 1 \\
3 & 1 & 1 & 1 & 1 & 0
\end{pmatrix}
\end{equation}
which is product-irregular.
\item For $(n,m,l) = (5,6,6)$ we can consider the same matrix as in \eqref{eq:M666} without first row (and column), i.e. without row (and column) $v$ such that $pd(v) = (0,0)$.
\end{enumerate}
For $(n,m,l) = (5,5,5)$ we will consider $\tilde{A}_5 \oplus \tilde{B}_5 \oplus \tilde{C}_5 +2edges$. Let $\tilde{B}_5$ to have $2$ in-edges, then we have:
\begin{enumerate}
\item[(1)] If $\tilde{B}_5$ has $2$ in-edges from one vertex, then we can take weighted adjacency matrix $\tilde{A}_5 \oplus \tilde{B}_5 \oplus \tilde{C}_5 + T_{12}(\tilde{A}_5, \tilde{B}_5, \tilde{C}_5, 3, 3, 3) + T_{23}(\tilde{A}_5, \tilde{B}_5, \tilde{C}_5, 3, 3, 2)$ which is product-irregular.
\item[(2)] If $\tilde{B}_5$ has $2$ in-edges from different vertices then we can take weighted adjacency matrix $\tilde{A}_5 \oplus \tilde{B}_5 \oplus \tilde{C}_5 + T_{12}(\tilde{A}_5, \tilde{B}_5, \tilde{C}_5, 3, 3, 3) + T_{23}(\tilde{A}_5, \tilde{B}_5, \tilde{C}_5, 1, 3, 2)$ which is product-irregular.
\end{enumerate}
The proof now follows by the above argumentation, together with Theorem \ref{t37} and Lemmas \ref{l366}, \ref{l_667}, \ref{l3567} and \ref{l3556}.
\end{proof}

\begin{lemma} \label{l3475}
For all positive integers $n \geq 7$ and $m \in \{ 5,6 \}$ we have that $ps(K_4 + K_n + K_m) = 3$.
\end{lemma}

\begin{proof}
Consider three different cases for different $m$:
\begin{enumerate}
\item For $m = 6$ and $n \geq 8$ consider matrix $A_4 \oplus B_6 \oplus B_n$ which is product-irregular using Theorem \ref{tmain1}.
\item For $m = 6$ and $n = 7$ consider matrix $A_4 \oplus B_7 \oplus \tilde{T_6}$ which is product-irregular (where $\tilde{T_6}$ is defined in \eqref{eq:T6}).
\item For $m = 5$ consider matrix $A_4 \oplus B_n \oplus \tilde{T_5}$ which is product-irregular (where $\tilde{T_5}$ is defined in \eqref{eq:T5}).
\end{enumerate}

\end{proof}

\begin{theorem} \label{tmain2}
For all positive integers n, m and l that are greater or equal than $4$ we have that $ps(K_n + K_m + K_l +2edges) = 3$.
\end{theorem}

\begin{proof}
For $(n,m,l) = (4,5,6)$ consider the following matrix:
\begin{equation}
A_4 \oplus \begin{pmatrix}
0 & 2 & 2 & 2 & 1 \\
2 & 0 & 3 & 1 & 3 \\
2 & 3 & 0 & 2 & 3 \\
2 & 1 & 2 & 0 & 1 \\
1 & 3 & 3 & 1 & 0
\end{pmatrix} \oplus B_6
\end{equation}
Notice that the second block of this matrix is $\tilde{T}_5$ from \eqref{eq:T6} in which we replaced elements $t_{24}$ and $t_{42}$ from $3$ to $1$.
\\Consider some cases for which we need to add some edges between cliques:
\begin{enumerate}
\item For $(n,m,l) = (4,5,5)$ we will consider $\tilde{A}_4 \oplus \tilde{B}_5 \oplus \tilde{C}_5 +2edges$.
\begin{enumerate}
\item[$(\tilde{B}_5)$] For the case when $d^+(\tilde{B}_5) = 2$ we have two options:
\begin{enumerate}
\item[(1)] If $\tilde{B}_5$ has $2$ in-edges from one vertex, then we can take weighted adjacency matrix $\tilde{A}_4 \oplus \tilde{B}_5 \oplus \tilde{C}_5 + T_{12}(\tilde{A}_4, \tilde{B}_5, \tilde{C}_5, 2, 3, 3) + T_{23}(\tilde{A}_4, \tilde{B}_5, \tilde{C}_5, 3, 3, 2)$ which is product-irregular.
\item[(2)] If $\tilde{B}_5$ has $2$ in-edges from different vertices then we can take weighted adjacency matrix $\tilde{A}_5 \oplus \tilde{B}_5 \oplus \tilde{C}_5 + T_{12}(\tilde{A}_4, \tilde{B}_5, \tilde{C}_5, 3, 3, 3) + T_{23}(\tilde{A}_4, \tilde{B}_5, \tilde{C}_5, 1, 3, 2)$ which is product-irregular.
\end{enumerate}
\item[$(\tilde{A}_4)$] For the case when $d^+(\tilde{A}_4) = 2$ we have two options:
\begin{enumerate}
\item[(1)] If $\tilde{A}_4$ has $2$ in-edges from one vertex then we can take weighted adjacency matrix $\tilde{A}_4 \oplus \tilde{B}_5 \oplus \tilde{C}_5 + T_{12}(\tilde{A}_4, \tilde{B}_5, \tilde{C}_5, 2, 3, 2) + T_{13}(\tilde{A}_4, \tilde{B}_5, \tilde{C}_5, 2, 3, 2)$ which is product-irregular.
\item[(2)] If $\tilde{A}_4$ has $2$ in-edges from different vertices then we can take weighted adjacency matrix $\tilde{A}_4 \oplus \tilde{B}_5 \oplus \tilde{C}_5 + T_{12}(\tilde{A}_4, \tilde{B}_5, \tilde{C}_5, 2, 3, 2) + T_{13}(\tilde{A}_4, \tilde{B}_5, \tilde{C}_5, 4, 3, 2)$ which is product-irregular.
\end{enumerate}
\end{enumerate}

\item For $(n,m,l) = (4,4,5)$ we will consider $\tilde{A}_4 \oplus \tilde{B}_5 \oplus \tilde{C}_4 +2edges$.
\begin{enumerate}
\item[$(\tilde{B}_5)$] For the case when $d^+(\tilde{B}_5) = 2$ we have two options:
\begin{enumerate}
\item[(1)] If $\tilde{B}_5$ has $2$ in-edges from one vertex then we can take weighted adjacency matrix $\tilde{A}_4 \oplus \tilde{B}_5 \oplus \tilde{C}_4 + T_{12}(\tilde{A}_4, \tilde{B}_5, \tilde{C}_4, 2, 3, 3) + T_{23}(\tilde{A}_4, \tilde{B}_5, \tilde{C}_4, 3, 2, 2)$ which is product-irregular.
\item[(2)] Since $ps(\tilde{A}_4 \oplus \tilde{B}_5 \oplus \tilde{C}_4 +2edges) > 3$ for this particular case, i.e. when $d^+(\tilde{B}_5) = 2$ and $\tilde{B}_5$ has $2$ in-edges from different vertices we have to consider $\tilde{A}_4 \oplus \tilde{B}_4 \oplus \tilde{C}_5 +2edges$ when $d^+(\tilde{C}_5) = 2$ and $\tilde{C}_5$ has $2$ in-edges from different vertices. Now we can take weighted adjacency matrix $\tilde{A}_4 \oplus \tilde{B}_4 \oplus \tilde{C}_5 + T_{13}(\tilde{A}_4, \tilde{B}_4, \tilde{C}_5, 2, 3, 3) + T_{23}(\tilde{A}_4, \tilde{B}_4, \tilde{C}_5, 2, 2, 3)$ which is product-irregular.
\end{enumerate}
\item[$(\tilde{C}_4)$] For the case when $d^+(\tilde{C}_4) = 2$ we have two options:
\begin{enumerate} 
\item[(1)] If $\tilde{C}_4$ has $2$ in-edges from one vertex then we can take weighted adjacency matrix $\tilde{A}_4 \oplus \tilde{B}_5 \oplus \tilde{C}_4 + T_{13}(\tilde{A}_4, \tilde{B}_5, \tilde{C}_4, 2, 2, 3) + T_{23}(\tilde{A}_4, \tilde{B}_5, \tilde{C}_4, 3, 2, 3)$ which is product-irregular.
\item[(2)] \label{itm:C4(2)} If $\tilde{C}_4$ has $2$ in-edges from different vertices then we can take weighted adjacency matrix $\tilde{A}_4 \oplus \tilde{B}_5 \oplus \tilde{C}_4 + T_{13}(\tilde{A}_4, \tilde{B}_5, \tilde{C}_4, 2, 2, 3) + T_{23}(\tilde{A}_4, \tilde{B}_5, \tilde{C}_4, 3, 1, 3)$ which is product-irregular.
\end{enumerate}
\end{enumerate}
\item For $(n,m,l) = (4,4,4)$ we will consider $\tilde{A}_4 \oplus \tilde{B}_4 \oplus \tilde{C}_4 +2edges$. Let $\tilde{C}_4$ to have $2$ in-edges, then we have the same proof as in item $(\tilde{C}_4)$\eqref{itm:C4(2)} replacing $\tilde{B}_5$ with $\tilde{B}_4$.
\end{enumerate}
The proof now follows by the above argumentation, together with Theorem \ref{t35e} and Lemma \ref{l3475}.
\end{proof}

\begin{corollary} \label{cmain2}
If G is a connected graph such that its vertex set can be partitioned into $3$ cliques of sizes at least $4$ then $ps(G) = 3$.
\end{corollary}

We would like to conclude the paper with proposing the following problem for possible further research.

\begin{problem}
Are there only finitely many connected graphs with clique cover number $4$ and product irregularity strength more than $3$?
\end{problem}

\end{document}